\documentclass[12pt,a4paper,article]{amsart}
\usepackage{amsmath, amsthm, verbatim, amsfonts, amssymb, color}
\usepackage{mathrsfs}
\usepackage{dsfont}
\usepackage{mathtools}
\usepackage[a4paper,margin=1in]{geometry}

\theoremstyle{plain}
\newtheorem{theorem}{Theorem}[section]
\newtheorem{corollary}[theorem]{Corollary}
\newtheorem{lemma}[theorem]{Lemma}
\newtheorem{proposition}[theorem]{Proposition}

\theoremstyle{definition}

\newtheorem*{h1}{(H1)}
\newtheorem*{h2}{(H2)}

\newtheorem*{ack}{Acknowledgement}
\numberwithin{equation}{section}
\renewcommand\labelenumi{\textup{\alph{enumi})}}
\renewcommand\theenumi\labelenumi

\newcommand\Ee{\mathds{E}}
\newcommand\Pp{\mathds{P}}
\newcommand\Ww{\mathds{W}}
\newcommand\Ss{\mathds{S}}
\newcommand\Dd{\mathds{D}}

\newcommand\real{{\mathds{R}}}

\newcommand\rd{{\mathds{R}^d}}

\newcommand\I{\mathds{1}}

\newcommand\dup{\mathrm{d}}
\newcommand\dif{\mathrm{d}}
\newcommand\eup{\mathrm{e}}
\newcommand\iup{\mathop{\mathrm{i}}}

\newcommand\Lip{\mathrm{Lip}}
\newcommand\RE{\operatorname{Re}}
\newcommand\Dom{\mathrm{Dom}}

\newcommand\HS{\mathrm{HS}}
\newcommand\law{\mathrm{law}}
\newcommand\lin{\mathrm{lin}}

\newcommand\op{\mathrm{op}}
\newcommand\TV{\mathrm{TV}}

\newcommand{\Ascr}{\mathscr{A}}
\newcommand\Bscr{\mathscr{B}}
\newcommand\Cscr{\mathscr{C}}
\newcommand{\Fscr}{\mathscr{F}}
\newcommand{\Dcal}{\mathcal{D}}
\newcommand{\Wcal}{\mathcal{W}}

\usepackage{url}

\usepackage{marginnote}

\begin{document}
\title[TV distance between stable and Brownian SDEs]{Total variation distance between SDEs with stable noise and Brownian motion}



\author[C.-S.~Deng]{Chang-Song Deng}
\address[C.-S.~Deng]{School of Mathematics and Statistics\\ Wuhan University\\ Wuhan 430072, China}
\email{dengcs@whu.edu.cn}

\author[X.~Li]{Xiang Li}
\address[X.~Li]{Department of Mathematics, Faculty of Science and Technology, University of Macau, Macau S.A.R., China }
\email{yc07904@um.edu.mo}

\author[R.L.~Schilling]{Ren\'{e} L.\ Schilling}
\address[R.L.\ Schilling]{%
TU Dresden, Fakult\"{a}t Mathematik, Institut f\"{u}r Ma\-the\-ma\-ti\-sche Sto\-cha\-stik, 01062 Dresden, Germany}
\email{rene.schilling@tu-dresden.de}

\author[L.~Xu]{Lihu Xu}
\address[L.~Xu]{Department of Mathematics, Faculty of Science and Technology, University of Macau, Macau S.A.R., China }
\email{lihuxu@um.edu.mo}

\begin{abstract}
We consider a $d$-dimensional stochastic differential equation (SDE) of the form $\dup U_t = b(U_t)\,\dup t + \sigma\,\dup Z_t$, let $X_t$ be the solution if the driving noise $Z_t$ is a $d$-dimensional rotationally symmetric $\alpha$-stable process ($1<\alpha<2$), and let $Y_t$ be the solution if the driving noise is a $d$-dimensional Brownian motion.

Continuing the work in \cite{DSX23}, we derive an estimate of the total variation distance $\|\law (X_{t})-\law (Y_{t})\|_\TV$ for all $t>0$, and we show that the ergodic measures $\mu_\alpha$ and $\mu_2$ of $X_t$ and $Y_t$, respectively, satisfy
$$\|\mu_\alpha-\mu_2\|_\TV \leq \frac{Cd\log(1+d)}{\alpha-1}(2-\alpha).$$
We shall show that this bound is optimal with respect to $\alpha$ by an Ornstein--Uhlenbeck SDE. 
Combining this bound with a recent interpolation result from \cite{HRW23}, we can derive a bound in Wasserstein-$p$ distance ($0< p <1$):
\begin{gather*}
       \|\mu_\alpha-\mu_2\|_{W_p} \leq\frac{Cd^{(p+3)/2}\log(1+d)}{\alpha-1} (2-\alpha).
\end{gather*}
{\bf Key Words:} Total variation distance, Wasserstein-$p$ distance, stochastic differential equation,
Poisson equation, stable process.
\end{abstract}

\maketitle

\tableofcontents

\noindent

\section{Introduction}
Let us consider the SDEs
\begin{gather}\label{stableSDE}
    \dup X_t = b(X_t)\,\dup t + \sigma\,\dup L_t, \quad X_0=x,
\\\label{BM-SDE}
    \dup Y_t = b(Y_t)\,\dup t + \sigma\,\dup B_t, \quad Y_0=y,
\end{gather}
driven by a standard $d$-dimensioal Brownian motion $B_t$ and a stable L\'evy process $L_t$ with characteristic function $\Ee\,\eup^{\iup \xi L_t}=\eup^{-t|\xi|^\alpha/2}$, $1<\alpha<2$; this is a standard rotationally symmetric $\alpha$-stable process run at half speed.
Our main assumptions are as follows:
\begin{h1}\label{h1}
$\sigma\in\real^{d\times d}$ is an invertible $d \times d$ matrix.
\end{h1}
\begin{h2}\label{h2}
$b \in C^2(\rd, \rd)$ and there exist constants $\theta_0>0$, $K \ge 0$, $\theta_1 \ge 0$, $\theta_2\geq 0$ such that 
for all $x,y,v,v_1,v_2\in \rd$,
\begin{align}
    \langle x-y, b(x)-b(y)\rangle&\leq-\theta_0|x-y|^2+K,\label{H1-1}\\
    |\nabla_v b(x)|
    &\leq\theta_1|v|, \label{H1-1'}\\
    |\nabla_{v_2}\nabla_{v_1} b(x)|
    &\leq\theta_2|v_1||v_2|. \label{H1-2}
\end{align}
\end{h2}
It is well known that under \eqref{H1-1'} the both SDEs \eqref{stableSDE} and \eqref{BM-SDE} have unique non-explosive (strong) solutions, which we denote by $X^x_t$ and $Y^y_t$ respectively.  From the classical Lyapunov function criterion, see e.g.\ \cite{MeTw09}, we know that under \textup{\textbf{(H1)}} and \textup{\textbf{(H2)}} the solutions to the SDEs \eqref{stableSDE} and \eqref{BM-SDE} are ergodic. Denote by $\mu_\alpha$ and $\mu_2$ the respective ergodic measures of $X^x_t$ and $Y^y_t$.


Continuing the work in Continuing the work in \cite{DSX23}, we establish an optimal bound between $X_t^x$ and $Y_t^y$ in total variation distance and consequently an optimal bound between $\mu_\alpha$ and $\mu_2$ as $t \rightarrow \infty$.  Combining this bound with a recent interpolation result from \cite{HRW23}, we can extend these bounds to the Wasserstein-$p$ distance with $0 < p < 1$.

\subsection{Notation} Before giving our main results, let us first introduce the notation that will be used in sequel.

We denote by $C(\rd, \real^m)$ the set of all continuous functions from $\rd$ to $\real^m$ and by $C_b(\rd, \real)$ the set of all bounded continuous functions from $\rd$ to $\real$. Given $k \in \mathds{N}$,  denote by $C^k(\rd, \real)$ the set of all continuous functions from $\rd$ to $\real$ with continuous $1$st,...,$k$-th order derivatives, and by $C^k_b(\rd, \real)$ the set of all bounded continuous functions from $\rd$ to $\real$ with bounded continuous $1$st,...,$k$-th order derivatives.
Let $f:\rd\to\real$ be sufficiently smooth, the directional derivative in direction $v\in\rd$ is defined as
\begin{gather*}
	\nabla_{v} f(x) := \lim_{\epsilon\to 0} \frac{f(x+\epsilon v)-f(x)}{\epsilon} = \nabla f(x)\cdot v
\end{gather*}
and $\nabla_{v_2}\nabla_{v_1} f(x) := \nabla_{v_2}(\nabla_{v_1} f)(x) = \nabla^2f(x){v_1}\cdot {v_2}$ for $v_1,v_2\in\rd$. The vector and matrix norms of $\nabla f$ and $\nabla^2 f$ are given by
\begin{gather*}
	|\nabla f(x)| = \sup_{v\in\rd, |v|=1}\left|\nabla_v f(x)\right|
\end{gather*}
and
\begin{gather*}
	\|\nabla^2 f(x)\|_\op = \sup_{v_1,v_2\in\rd, |v_1|=|v_2|=1}\left|\nabla_{v_2}\nabla_{v_1} f(x)\right|
=\sup_{v_1,v_2\in\rd, |v_1|=|v_2|=1}\left\langle\nabla^2f(x),v_1v_2^\top\right\rangle_{\rm HS}.
\end{gather*}
Moreover, we define
$$\|\nabla f\|_\infty:=\sup_{x \in \rd} \|\nabla f(x)\|_\op, \ \ \ \|\nabla^2 f\|_{\op,\infty}:=\sup_{x \in \rd} \|\nabla^2 f(x)\|_\op.$$ This notation extends naturally to vector-valued functions $f=(f_{1},f_{2},\cdots,f_{d})^{\top}: \rd\to \rd$:
\begin{align*}
	\nabla f(x)v&=\nabla_{v}f(x)=(\nabla_{v}f_{1}(x),\dots,\nabla_{v}f_{d}(x))^{\top},\\
	\nabla^2f(x)v_1v_2&=\nabla_{v_2}\nabla_{v_1}f(x)= (\nabla_{v_2}\nabla_{v_1}f_{1}(x),\dots,\nabla_{v_2} \nabla_{v_1}f_{d}(x))^{\top}.
\end{align*}
For $x \in \rd$, we see $\nabla f(x) \in \rd$ and $\nabla^2 f(x) \in \real^{d \times d}$. The Hilbert--Schmidt inner product of two matrices $A, B \in \real^{d \times d}$ is $\left\langle A, B\right\rangle_{\mathrm{HS}}=\sum_{i,j=1}^d A_{ij} B_{ij}$, and the Hilbert-Schmidt norm is $\|A\|_{\mathrm{HS}}=\sqrt{\sum_{i,j=1}^dA_{ij}^2}$.

Recall that $X_t^x$ and $Y_t^y$ are the solutions to SDEs \eqref{stableSDE} and \eqref{BM-SDE}  respectively. 
The operator semigroup induced by the Markov process $(X_t^x)_{t\geq 0}$ is
\begin{gather*}
    P_t f(x) = \Ee f(X^x_t),\quad f \in C_b(\rd,\real),\, t>0.
\end{gather*}
We need to study $P_t$ on the Banach space $\left(C_{\lin}(\rd,\real),\|\cdot\|_\lin\right)$ of linearly growing continuous functions
\begin{gather*}
 C_{\lin}(\rd,\real)
 = \left\{f \in C(\rd,\real)\,;\, \|f\|_\lin := \sup_{x \in \rd} \frac{|f(x)|}{1+|x|}<\infty\right\}.
\end{gather*}
By \cite{DSX23}, the semigroup $(P_t)_{t \ge 0}$ is well defined on $C_{\lin}(\rd,\real)$. The infinitesimal generator $\Ascr^P$ of the semigroup $(P_t)_{t\geq 0}$ has the following domain:
\begin{align*}
 \Dcal(\Ascr^P)
 :=\bigg\{f \in C_{\lin}(\rd,\real)\,;\,& g(x)=  \lim_{t\to 0} \frac{P_t f(x)-f(x)}{t} \text{\ \ exists \ for all $x\in\rd$}\\
 & \text{and\ } g \in C_{\lin}(\rd,\real)\bigg\}.
\end{align*}
In the same way we can consider the semigroup $Q_t$ of $(Y_t^y)_{t\geq 0}$ and its generator $\left(\Ascr^Q, \Dcal(\Ascr^Q)\right)$.

Write $\Cscr_b(1) = \{h\in C_b(\rd,\real)\:;\: \|h\|_\infty \leq 1\}$ for the unit ball in $C_b(\rd,\real)$. The total variation distance of two probability measures (on $\rd$) $\mu$ and $\nu$  is given by
\begin{gather*}
   \|\mu-\nu\|_\TV := \sup_{h\in\Cscr_b(1)}|\mu(h)-\nu(h)|
\end{gather*}
(as usual we write $\mu(h):=\int h\,\dup\mu$ etc.). It is clear that $\|\mu-\nu\|_\TV\leq2$ for all probability measures $\mu$ and $\nu$.
 For $p\in(0,1]$, the Wasserstein-$p$ distance for two probability measures $\mu,\nu$  is given by
   \begin{gather*}
        \|\mu-\nu\|_{W_p}=\sup_{|h(x)-h(y)|\leq|x-y|^p}|\mu(h)-\nu(h)|.
   \end{gather*}

Finally, $\Lip(1)$ is the family of Lipschitz functions $f:\rd\to\real$ with Lipschitz constant $1$ and norm
$\|f\|_{\Lip}:=\sup_{x \neq y}\frac{|f(x)-f(y)|}{|x-y|} \le 1$.
\subsection{Main results}

Throughout this paper, $C,C_1,C_2$ denote positive constants which may depend on $\theta_0,\theta_1, \theta_2,K, \|\sigma\|_{\op},\|\sigma^{-1}\|_{\op}$, but they are always independent of $d$ and $\alpha$; their values may change from line to line. Recall that $\mu_\alpha$ and $\mu_2$ are respective ergodic measures of $X^x_t$ and $Y^y_t$ respectively.

\begin{theorem}\label{main1}
	Assume \textup{\textbf{(H1)}} and \textup{\textbf{(H2)}}. For any $\alpha \in (1,2)$, $x,y\in\rd$ and $t>0$,
	\begin{align*}
    	&\|\law (X_{t}^x)-\law (Y_{t}^y)\|_\TV\\
    	&\qquad\leq C_{1} \left\{t^{-1/2} \vee 1\right\}\eup^{-C_{2}t}|x-y| +\frac{Cd\log(1+d)}{\alpha-1}(2-\alpha)\cdot
        \left\{\left(t^{-\frac{1}{\alpha}+\frac{1}{2}}\log\tfrac{1}{t}\right) \vee 1  \right\}.
	\end{align*}
	In particular,
	\begin{gather}\label{was-est}
    	\|\mu_\alpha-\mu_2\|_\TV
    	\leq \frac{Cd\log(1+d)}{\alpha-1}(2-\alpha).
	\end{gather}
\end{theorem}

We shall show that the bound in \eqref{was-est} is optimal with respect to $\alpha$ by a one-dimensional Ornstein--Uhlenbeck system, cf.\ Proposition~\ref{ourate} below.

Applying a recent interpolation result established in \cite{HRW23}, we obtain the following corollary.
\begin{corollary} \label{corollary}
Assume \textup{\textbf{(H1)}} and \textup{\textbf{(H2)}}. For any $p\in(0,1)$ and $\alpha\in(1,2)$,
 \begin{gather*}
      \|\mu_\alpha-\mu_2\|_{W_p} \leq\frac{Cd^{(p+3)/2}\log(1+d)}{\alpha-1} (2-\alpha).
   \end{gather*}
\end{corollary}


\section{A key gradient estimate for Theorem \ref{main1}}

This section is devoted to proving the following gradient estimate by an integration by parts in Malliavin calculus, which is crucial for the proof of our main result.
\begin{proposition} \label{2ndgrad}
    Assume \eqref{H1-1'} and \eqref{H1-2}. For any $t\in(0,1]$, $\alpha\in(0,2)$ and $f\in C^{2}_{b}(\rd,\real)$,
    \begin{align*}
        \|P_t\nabla^2f\|_{\op,\infty}
        \le C\sqrt{\Gamma\left(1+\tfrac{2}{\alpha}\right)} \,2^{1/\alpha}t^{-1/\alpha}\|\nabla f\|_{\infty}.
    \end{align*}
\end{proposition}

\subsection{A random time-change method}
This method was originally proposed by X.~Zhang \cite{Zha13} and, in a similar context used in our publication \cite{CDSX23}. In order to keep the presentation self-contained, we recall the basic ingredients here. Using Bochner's subordination, we can in law represent the stable L\'evy process $(L_t)_{t\geq 0}$ with characteristic function $\exp(-\frac 12 t |\xi|^\alpha)$ as a time-changed Brownian motion $(W_{S_t})_{t\geq 0}$, where $(W_t)_{t\geq 0}$ is a standard Brownian motion with characteristic function $\exp(-\frac 12 t |\xi|^2)$ and $(S_t)_{t\geq 0}$ is an independent stable subordinator (i.e.\ an $\real^+$-valued L\'evy process) withthe following Laplace transform:
\begin{gather*}
    \Ee\left[\eup^{-rS_t}\right] = \eup^{-2^{-1}t (2r)^{\alpha/2}},\quad r>0,\,t\geq 0,
\end{gather*}
see e.g.\ \cite{CKSV22,SSZ12, Zha13}. Since we are interested only in distributional properties of the solution to the SDE \eqref{stableSDE}, we may replace it by
\begin{align}\label{stableSDE-2}
	\dif X_{t}=b(X_{t})\,\dif t+\sigma\,\dif W_{S_{t}}, \quad X_{0}=x.
\end{align}
This allows us to work with the canonical realizations of Brownian motion and the subordinator on the product probability space $\left(\Ww\times\Ss,\Bscr(\Ww)\otimes\Bscr(\Ss), \mu_{\Ww}\times\mu_{\Ss}\right)$, where $\Ww$ and $\Ss$ are the canonical path spaces of Brownian motion: $\Ww$ are the continuous functions $w:[0,\infty)\to\rd$ with $w(0)=0$ and $\Ss$ are the right-continuous increasing functions $\ell:[0,\infty) \to [0,\infty)$ with $\ell(0)=0$. In particular, we can calculate the expectation of the solution $X_{t}^{x}$ to the SDE \eqref{stableSDE-2} as
\begin{gather}\label{e:subord}
    \Ee f\left(X_{t}^{x}\right)
    =\int_{\Ss}\int_{\Ww} f\left(X_{t}^{x}(w\circ\ell)\right)\mu_{\Ww}(\dif w) \, \mu_{\Ss}(\dif \ell).
\end{gather}
This means that we can freeze the path of the subordinator and do all calculations for a Brownian motion which is time-changed with a deterministic time-change $\ell = (\ell_t)_{t\geq 0}\in\Ss$. After that, we only  need to make sure that our results remain stable under the integration with respect to $\mu_{\Ss}(\dif \ell)$. If needed, we write $\Ee^\Ww$ and $\Ee^\Ss$ for the expectation taken only in the $\Ww$ or $\Ss$ component.

For fixed $\ell\in\Ss$, denote by $X_{t}^{\ell;x}$ the solution to the SDE
\begin{align}\label{ellSDE}
    \dif X_{t}^{x;\ell}=b(X_{t}^{x;\ell})\,\dif t + \sigma\,\dif W_{\ell_{t}-\ell_{0}},
    \quad X_{0}^{x;\ell}=x.
\end{align}
By construction, $W_{\ell_{t}}$ is a c\`adl\`ag $\Fscr_{\ell_{t}}^{\Ww}$-martingale. Thus, \textup{\textbf{(H1)}} and \textup{\textbf{(H2)}} guarantee that the SDE \eqref{ellSDE} admits a unique c\`adl\`ag $\Fscr_{\ell_{t}}^{\Ww}$-adapted solution $X_{t}^{x;\ell}$ for each initial value $x\in\rd$.

For technical reasons, we have to regularize $\ell\in\Ss$. Fix $\epsilon\in(0,1)$ and define
\begin{align*}
    \ell_{t}^{\epsilon}
    := \frac{1}{\epsilon}\int_{t}^{t+\epsilon}\ell_{s}\,\dif s + \epsilon t
    = \int_{0}^{1}\ell_{\epsilon s+t}\,\dif s + \epsilon t.
\end{align*}
It is not hard to see that $\ell_t^\epsilon\downarrow \ell_t$ ($t\geq 0$) as $\epsilon\downarrow 0$ and $t\mapsto \ell_{t}^{\epsilon}$ is absolutely continuous and strictly increasing. In slight abuse of notation we write $X_{t}^{x;\ell^\epsilon}$ for the solution to the SDE \eqref{ellSDE} with $\ell$ replaced by $\ell^\epsilon$.

Let $\gamma^{\epsilon}$ be the inverse function of $\ell^{\epsilon}$, i.e., $\ell^\epsilon\circ\gamma^{\epsilon}_{t} = t$ for all $t\geq \ell_{0}^{\epsilon}$ and $\gamma^\epsilon\circ\ell^{\epsilon}_t=t$ for all $t\geq 0$. By construction, $t\mapsto\gamma_{t}^{\epsilon}$ is absolutely continuous on $[\ell_{0}^{\epsilon},\infty)$. Define
\begin{gather*}
    Z_{t}^{x;\ell^{\epsilon}}
    := X_{\gamma_{t}^{\varepsilon}}^{x;\ell^{\epsilon}}, \quad t\geq \ell_{0}^{\epsilon}.
\end{gather*}
From \eqref{ellSDE} (with $\ell\rightsquigarrow\ell^\epsilon$) and with a change of variables, we see for $t\geq \ell_{0}^{\epsilon}$,
\begin{equation}\label{Z}
    Z_{t}^{x;\ell^{\epsilon}}
    = x+\int_{0}^{\gamma_{t}^{\epsilon}} b\left(X_{s}^{x;\ell^{\epsilon}}\right) \dif s + \sigma W_{t-\ell^\epsilon_0}
    = x+\int_{\ell_{0}^{\epsilon}}^{t} b\left(Z_{s}^{x;\ell^{\epsilon}}\right)\dot{\gamma}_{s}^{\epsilon}\,\dif s+\sigma W_{t-\ell^\epsilon_0}.
\end{equation}

\subsection{Integration by parts via Malliavin Calculus}
For fixed $t>\ell^{\epsilon}_{0}$, consider the Hilbert space $H:=L^{2}([\ell^{\epsilon}_{0},t];\rd)$ equipped with the inner product
\begin{gather*}
    \langle f, g \rangle_{H}:= \int_{\ell^{\epsilon}_{0}}^{t} \langle f(s),g(s) \rangle\,\dup s\quad  \text{for $f,g\in H$}.
\intertext{Let $W=\{W(h)\,;\,h\in H \}$ be an isonormal Gaussian process associated with $H$, i.e.}
    \Ee^{\Ww}[W(f)W(g)]=\langle f, g \rangle_{H},\quad f,g\in H,
\end{gather*}
which is defined on the Wiener space $\left(\Ww,\Bscr(\Ww),\mu_{\Ww}\right)$. For the $i$th standard coordinate vector $e_{i}\in\rd$, $i=1,2,\dots,d$, it is clear that $W(e_{1}\I_{[\ell^{\epsilon}_{0},t]}(s)),\dots ,W(e_{d}\I_{[\ell^{\epsilon}_{0},t]}(s))$ are pairwise independent one-dimensional standard Brownian motions.

Let $D$ denote the Malliavin derivative; the directional Malliavin derivative of a random variable $F\in \Dd^{1,2}$ (the domain of
$D$ in $L^{2}(\Ww,\Bscr(\Ww),\mu_{\Ww})$) along $u\in H$ is
\begin{gather*}
	D_{u}F:=\langle DF, u \rangle_{H}
	= \int_{\ell^{\epsilon}_{0}}^t \langle  D F(s),u(s)\rangle\,\dif s
\end{gather*}
with  $DF \in L^2(\Ww, H)$. For a $d$-dimensional random vector $F=(F_{1},\dots, F_{d})^{\top}$ with $F_{i}\in \Dd^{1,2}$ for
each $1\leq i\leq d$, we set $DF:=(DF_{1},\dots DF_{d})^{\top}$ and $D_{u}F= \left(D_{u}F_{1}, \dots, D_{u} F_{d} \right)^{\top}$.

Let $\delta$ be the divergence operator, i.e.\ the adjoint of $D$. It is well known that, for any $u\in \Dom(\delta)$ and $F\in \mathds{D}^{1,2}$, $\delta(u)$ satisfies
\begin{align}\label{duality}
    \Ee^\Ww\left[F\delta(u) \right]=\Ee^\Ww[\langle DF,u \rangle_{H}].
\end{align}
The following auxiliary lemma is crucial for proving Proposition \ref{2ndgrad}.
\begin{lemma}\label{graepsi}
	Assume \eqref{H1-1'} and \eqref{H1-2}, and fix $\ell\in\Ss$ and $\epsilon\in(0,1)$. For any $v\in\real^d$, $t>0$, $f\in C^{1}_{b}( \rd,\real)$ and $x\in\rd$,
    \begin{align*}
    \left| \Ee^\Ww \left[\nabla_v f\left(X^{x;\ell^{\epsilon}}_{t} \right)\right]\right|
    \leq C|v|\left( \frac{1}{\sqrt{\ell^{\epsilon}_{t}-\ell^{\epsilon}_{0}}} + t\right)
    \eup^{2\theta_{1} t} \sqrt{\Ee^\Ww\left[\left|f(X^{x;\ell^{\epsilon}}_{t}) \right|^{2}\right]}.
    \end{align*}
\end{lemma}

\begin{proof}
By \eqref{Z}, the Malliavin derivative of $Z^{x;\ell^{\epsilon}}_{t}$ along the direction $u\in H$ satisfies
\begin{gather*}
    \dif D_{u}Z^{x;\ell^{\epsilon}}_{t}
    = \left[\nabla b\left(Z^{x;\ell^{\epsilon}}_{t}\right) \dot{\gamma}^{\epsilon}_{t}D_{u}Z^{x;\ell^{\epsilon}}_{t}+\sigma u(t) \right]\,\dif t,
    \quad t\ge \ell^{\epsilon}_{0}.
\end{gather*}
Since $D_{u}Z^{x;\ell^{\epsilon}}_{\ell^{\epsilon}_{0}}=0$, this equation has a solution of the form
\begin{gather} \label{DZ}
    D_{u}Z^{x;\ell^{\epsilon}}_{t}
    = \int_{\ell^{\epsilon}_{0}}^{t}  J_{s,t}\sigma  u(s)\,\dif s,
\end{gather}
where
\begin{gather*}
    J_{s,t}
    := \exp\left(\int_{s}^{t} \nabla b\left(Z^{x;\ell^{\epsilon}}_{r}\right) \dot{\gamma}^{\epsilon}_{r} \,\dif r \right),
    \quad \ell^{\epsilon}_{0}\leq s\leq t.
\end{gather*}
For any fixed $t>\ell^{\epsilon}_{0}$ and $v\in\real^d$, let
\begin{gather} \label{ui}
    u_0(s):=\frac{1}{t-\ell^{\epsilon}_{0}}\,\sigma^{-1}J_{s,t}^{-1}v.
\end{gather}
It can be easily verified that
\begin{gather*}
    D_{u_{0}}Z^{x;\ell^{\epsilon}}_{t}=v.
\end{gather*}
Combining this with the duality \eqref{duality} between the Malliavin derivative and the divergence operator gives us, for any $ f\in C^{1}_{b}(\rd,\real)$,
\begin{align*}
    \Ee^\Ww\left[\nabla_v f\left(Z^{x;\ell^{\epsilon}}_{t} \right)\right]
    &=\Ee^\Ww\left[\big\langle \nabla f\left(Z^{x;\ell^{\epsilon}}_{t} \right),v \big\rangle \right]
    =\Ee^\Ww\left[\big\langle \nabla f\left(Z^{x;\ell^{\epsilon}}_{t} \right), D_{u_{0}}Z^{x;\ell^{\epsilon}}_{t} \big\rangle \right]\\
    &=\Ee^\Ww\left[D_{u_{0}}f\left(Z^{x;\ell^{\epsilon}}_{t} \right)  \right]
    =\Ee^\Ww\left[f\left(Z^{x;\ell^{\epsilon}}_{t} \right) \delta\left( u_{0}\right) \right].
\end{align*}
By the Cauchy--Schwarz inequality and \cite[Proposition 1.3.1]{N06}, we obtain
\begin{align}\label{gradvf}
\begin{split}
    \left| \Ee^\Ww\left[ \nabla_v f\left(Z^{x;\ell^{\epsilon}}_{t} \right)\right]\right|
    &\leq \sqrt{\Ee^\Ww\left[\left|f\left(Z^{x;\ell^{\epsilon}}_{t}\right) \right|^{2}\right]}
    \sqrt{\Ee^\Ww\left[ \delta(u_{0})^{2}\right]}\\
    &\leq \sqrt{\Ee^\Ww\left[\left|f\left(Z^{x;\ell^{\epsilon}}_{t}\right)\right|^{2}\right]}
    \sqrt{\Ee^\Ww\left[\|u_{0}\|_{H}^{2}+\|Du_{0}\|_{H\otimes H}^{2}\right]},
\end{split}
\end{align}
where $\|\cdot\|_{H\otimes H}$ denotes the norm on the tensor space $H\otimes H$. By \eqref{H1-1'} and \eqref{ui}, one has
\begin{align} \label{Unorm}
    \begin{split}
        \|u_{0}\|_{H}^{2}
    &=\frac{1}{(t-\ell^{\epsilon}_{0})^{2}}\int_{\ell^{\epsilon}_{0}}^{t} \left| \sigma^{-1}\exp\left(-\int_{s}^{t} \nabla b\left(Z^{x;\ell^{\epsilon}}_{r}\right) \dot{\gamma}^{\epsilon}_{r} \,\dif r \right)v \right|^{2} \dif s\\
    &\le \frac{\|\sigma^{-1}\|_{\op}^{2}|v|^2}{(t-\ell^{\epsilon}_{0})^{2}}\int_{\ell^{\epsilon}_{0}}^{t} \exp\left(2\theta_{1}\int_{s}^{t}  \dot{\gamma}^{\epsilon}_{r}\,\dif r \right) \dif s\\
    &\le \frac{\|\sigma^{-1}\|_{\op}^{2}|v|^2}{t-\ell^{\epsilon}_{0}} \,\eup^{2\theta_{1}\gamma^{\epsilon}_{t}}.
    \end{split}
\end{align}
Similarly, using \eqref{H1-1'} and \eqref{DZ} gives for any $u\in H$
\begin{align}\label{DVZ}
    \left| D_{u}Z^{x;\ell^{\epsilon}}_{t}\right|
    \leq \|\sigma\|_{\op} \|u\|_{H}\eup^{\theta_{1} \gamma^{\epsilon}_{t}} \sqrt{t-\ell^{\epsilon}_{0}}.
\end{align}
Furthermore, by the chain rule for matrix exponential functions (see e.g.\ \cite{N-H}), one has
\begin{align*}
    \begin{split}
         D_{u}J_{s,t}^{-1}
    &=\int_{0}^{1} \exp\left(-\tau\int_{s}^{t} \nabla b\left(Z^{x;\ell^{\epsilon}}_{r}\right) \dot{\gamma}^{\epsilon}_{r}\,\dif r \right)  \left(-\int_{s}^{t}\nabla^{2}b\left(Z^{x;\ell^{\epsilon}}_{r}\right) D_{u}Z^{x;\ell^{\epsilon}}_{r} \dot{\gamma}^{\epsilon}_{r} \,\dif r\right)\\
    &\qquad\qquad\qquad\qquad\qquad\qquad\qquad\quad\mbox{}\times \exp\left(-(1-\tau)\int_{s}^{t} \nabla b\left(Z^{x;\ell^{\epsilon}}_{r}\right) \dot{\gamma}^{\epsilon}_{r}\, \dif r \right)\dif \tau,
    \end{split}
\end{align*}
which, together with \eqref{H1-1'}, \eqref{H1-2} and \eqref{DVZ}, yields that
\begin{align*}
    \|D_{u}J_{s,t}^{-1}\|_{\op,\infty}\le \|\sigma\|_{\op}\theta_{2}\|u\|_{H}\sqrt{t-\ell^{\epsilon}_{0}}\gamma^{\epsilon}_{t}
    \eup^{2\theta_{1}\gamma^{\epsilon}_{t}}, \quad s\in [\ell^{\epsilon}_{0},t].
\end{align*}
Then, it follows from \eqref{ui} that
\begin{align*}
        \|D_{u}u_{0}\|^{2}_{H}
    &=\frac{1}{(t-\ell^{\epsilon}_{0})^{2}}\int_{\ell^{\epsilon}_0}^{t}  \left| \sigma^{-1}\left\{D_{u}J_{s,t}^{-1}\right\}v \right|^{2} \dif s\\
    &\le \|\sigma^{-1}\|_{\op}^2\|\sigma\|_{\op}^2\theta_{2}^{2}\|u\|_{H}^{2}(\gamma^{\epsilon}_{t})^{2}
    \eup^{4\theta_{1} \gamma^{\epsilon}_{t}}|v|^2.
   \end{align*}
Since this holds for any $u\in H$, one can deduce that
\begin{align*}
    \|Du_{0}\|^{2}_{H\otimes H}\le \|\sigma^{-1}\|_{\op}^2\|\sigma\|_{\op}^2\theta_{2}^{2}(\gamma^{\epsilon}_{t})^{2}
    \eup^{4\theta_{1} \gamma^{\epsilon}_{t}}|v|^2.
\end{align*}
Combining this with \eqref{gradvf} and \eqref{Unorm} gives us for $t\ge \ell^{\epsilon}_{0}$,
\begin{align*}
    &\left| \Ee^\Ww \left[\nabla_{v} f\left(Z^{x;\ell^{\epsilon}}_{t} \right)\right]\right|\\
    &\quad\le
    \left(\frac{\|\sigma^{-1}\|_{\op}|v|}{\sqrt{t-\ell^{\epsilon}_{0}}}\,\eup^{\theta_{1} \gamma^{\epsilon}_{t}}
    +\|\sigma^{-1}\|_{\op}\|\sigma\|_{\op}\theta_{2}\gamma^{\epsilon}_{t}
    \eup^{2\theta_{1} \gamma^{\epsilon}_{t}}|v|
    \right)\sqrt{\Ee^\Ww\left[\big|f\left(Z^{x;\ell^{\epsilon}}_{t}\right) \big|^{2}\right]}\\
    &\quad\leq C|v|\left(\frac{1}{\sqrt{t-\ell^{\epsilon}_{0}}}+\gamma^{\epsilon}_{t}
    \right)\eup^{2\theta_{1} \gamma^{\epsilon}_{t}}
    \sqrt{\Ee^\Ww\left[\big|f\left(Z^{x;\ell^{\epsilon}}_{t}\right) \big|^{2}\right]}.
\end{align*}
Since $Z_{\ell_t^\epsilon}^{x;\ell^\epsilon}=X_{t}^{x;\ell^\epsilon}$, we can replace $t$ by $\ell_t^\epsilon$ and get the desired estimate.
\end{proof}

\subsection{Proof of Proposition \ref{2ndgrad}}

\begin{lemma}\label{grafixed}
	Assume \eqref{H1-1'} and \eqref{H1-2}. For any $\alpha\in(0,2)$, $t\in(0,1]$ and $f\in C^{1}_{b}( \rd,\real)$,
    \begin{align*}
        \left\| P_{t}\nabla f\right\|_\infty
        \leq C\sqrt{\Gamma\left(1+\tfrac{2}{\alpha}\right)} \,2^{1/\alpha}t^{-1/\alpha}
    	\|f\|_\infty.
    \end{align*}
\end{lemma}

\begin{proof}
Letting $\epsilon\downarrow0$ in Lemma \ref{graepsi} and using $\ell^\epsilon_t\downarrow\ell_t$ and \cite[Lemma 2.2]{Zha13}, we get
for $t>0$,
\begin{align*}
    \left| \Ee^\Ww \left[\nabla_{v} f\left(X^{x;\ell}_{t} \right)\right]\right|
    \le
    C|v|\left(\ell_t^{-1/2}+t\right) \eup^{2\theta_{1} t} \sqrt{\Ee^\Ww\left[\left|f(X^{x;\ell}_{t}) \right|^{2}\right]}.
\end{align*}
Now we ``unfreeze'' the fixed path $\ell$ and use \eqref{e:subord}. From the Cauchy--Schwarz inequality
we get for all $t>0$, $v\in\real^d$ and $x\in\real^d$,
\begin{align*}
    \left|P_t(\nabla_vf)(x)\right|
    &=\left|
    \int_{\Ss}\Ee^\Ww \left[\nabla_{v} f\left(X^{x;\ell}_{t} \right)\right]
     \,\mu_{\Ss}(\dif\ell)
    \right|\\
    &\leq\int_{\Ss}\big| \Ee^\Ww \left[\nabla_{v} f\left(X^{x;\ell}_{t} \right)\right]\big|
    \,\mu_{\Ss}(\dif\ell)\\
    &\leq C|v|\eup^{2\theta_{1} t}\int_{\Ss}\left(\ell_t^{-1/2}+t
    \right)\sqrt{\Ee^\Ww\left[\big|f(X^{x;\ell}_{t}) \big|^{2}\right]}
    \,\mu_{\Ss}(\dif\ell)\\
    &\leq C|v|\eup^{2\theta_{1} t}\sqrt{\int_{\Ss}\left(\ell_t^{-1/2}+t
    \right)^2\,\mu_{\Ss}(\dif\ell)}\,
    \sqrt{\int_{\Ss}\Ee^\Ww\left[\big|f(X^{x;\ell}_{t}) \big|^{2}\right]
    \,\mu_{\Ss}(\dif\ell)}\\
    &\leq C|v|\eup^{2\theta_{1} t}\sqrt{\int_{\Ss}\left(\ell_t^{-1}+t^2\right)\,\mu_{\Ss}(\dif\ell)}
    \,\sqrt{\Ee\left[\big|f(X^{x}_{t}) \big|^{2}\right]}\\
    &=C|v|\eup^{2\theta_{1} t}\sqrt{\Ee\left[S_t^{-1}\right]+t^2}\,\sqrt{P_t|f|^2(x)}\\
    &\leq C|v|\eup^{2\theta_{1} t}\left(\sqrt{\Ee\left[S_t^{-1}\right]}+t\right)\|f\|_\infty.
\end{align*}
From \cite[Lemma 4.1]{DS19} we know the following moment identity for $\alpha\in(0,2)$
\begin{equation}\label{smoment}
    \Ee \left[S_t^{-1}\right]
    = \frac12\,\Gamma\left(1+\tfrac{2}{\alpha}\right)2^{2/\alpha}
    t^{-2/\alpha},
    \quad t>0.
\end{equation}
Substituting this identity into the previous inequality, we obtain the claimed estimate for $t\in(0,1]$.
\end{proof}

\begin{proof}[\bfseries Proof of Proposition \ref{2ndgrad}]
    It follows from Lemma \ref{grafixed} that for
    all $t\in(0,1]$ and $\alpha\in(0,2)$,
    \begin{align*}
        \|P_t\nabla^2f\|_{\op,\infty}
        &=\sup_{x\in\real^d}\sup_{v_1,v_2\in\real^d,|v_1|=|v_2|=1}
        \left|P_t\nabla_{v_2}\nabla_{v_1}f(x)\right|\\
        &\leq C\sqrt{\Gamma\left(1+\tfrac{2}{\alpha}\right)} \,2^{1/\alpha}t^{-1/\alpha}
    	\sup_{v_1\in\real^d,|v_1|=1}
         \|\nabla_{v_1}f\|_\infty.
    \end{align*}
    Finally,
    \begin{gather*}
        \sup_{v\in\real^d,|v|=1}\|\nabla_vf\|_\infty=
        \sup_{x\in\real^d}\sup_{v\in\real^d,|v|=1}\left|\langle\nabla f(x),v\rangle\right|
        =\|\nabla f\|_\infty,
    \end{gather*}
   and this completes the proof.
\end{proof}

\section{Gradient estimates of $Q_t$ and comparison of the generators}

In this section we collect further tools for the proof of Theorem~\ref{main1}. Recall that $\Cscr_b(1)$ is the unit ball in $C_b(\rd,\real)$.
\begin{lemma}\label{TV-gradient}
    Assume \textup{\textbf{(H1)}} and \textup{\textbf{(H2)}}.  For all $h\in \Cscr_b(1)$ and $t> 0$,
    \begin{align}
        \|\nabla Q_{t}h\|_{\infty}&\le C_{1} \left\{t^{-1/2} \vee 1\right\} \eup^{-C_{2}t},\label{grad11}\\
        \|\nabla^{2}Q_{t}h\|_{\op,\infty} &\le C_{1} \left\{t^{-1} \vee 1\right\}\eup^{-C_{2}t},\label{grad22}
    \end{align}
    where both $C_1$ and $C_2$ do not depend on the dimension $d$.
\end{lemma}
\begin{proof}
	According to \cite[Corollary 2]{E16}, for any $t>0$ and $x,y\in \rd$,
    \begin{gather*}
        \sup_{f\in\Lip(1)}\left|Q_{t}f(x)-Q_{t}f(y)\right|
        \le C_{1}\eup^{-C_{2}t}|x-y|,
    \end{gather*}
    which means that for any $f\in \Lip(1)$,
    \begin{gather*}
        \| \nabla Q_{t}f \|_{\infty}\le C_{1}\eup^{-C_{2}t} \quad \text{for all $t>0$}.
    \end{gather*}
    On the other hand, it follows from \cite{P-W} that for all $h\in \Cscr_b(1)$ and $t\in (0,1]$,
    \begin{gather*}
        \|\nabla Q_{t}h\|_{\infty}\le C t^{-1/2},
    \end{gather*}
    which implies that $Q_{t}h$ is a Lipschitz function with Lipschitz constant $C t^{-1/2}$. Hence, if $t\in (0,1]$, \eqref{grad11} holds. If $t>1$, we have
    \begin{gather*}
        \|\nabla Q_{t}h\|_{\infty}
        =\|\nabla Q_{t-1}\{Q_{1}h\}\|_{\infty}
        \le CC_1\eup^{-C_2(t-1)}=CC_1\eup^{C_2}\eup^{-C_2t},
    \end{gather*}
    which yields \eqref{grad11}.

    To prove \eqref{grad22}, we first note that by \cite[Lemma 2.1, (2.2)]{DSX23}, for any $f\in \Lip(1)$ and $t\in (0,1]$,
    \begin{gather*}
        \|\nabla^{2} Q_{t}f\|_{\op,\infty}\le C t^{-1/2}.
    \end{gather*}
    So by \eqref{grad11}, if $t\in (0,2]$, it holds that for $h\in \Cscr_b(1)$
    \begin{align*}
        \|\nabla^{2}Q_{t}h\|_{\op,\infty}
        = \|\nabla^{2} Q_{t/2}\{ Q_{t/2}h\}\|_{\op,\infty}
        & \leq C_1(t/2)^{-1/2}\eup^{-C_2t/2}C(t/2)^{-1/2}\\
        &= 2C_1Ct^{-1}\eup^{-C_2t/2},
    \end{align*}
    while for $t>2$ we have for $h\in \Cscr_b(1)$
    \begin{gather*}
        \|\nabla^{2}Q_{t}h\|_{\op,\infty} = \|\nabla^{2} Q_{1}\{ Q_{t-1}h\}\|_{\op,\infty}\le
        C_1\eup^{-C_2(t-1)}C=C_1C\eup^{C_2}\eup^{-C_2t}.
    \end{gather*}
    This completes the proof of \eqref{grad22}.
\end{proof}

In the following lemma, we will frequently need the constants
\begin{gather*}
    A(d,\alpha):=\frac{\alpha\Gamma(\frac{d+\alpha}{2})}
    {2^{2-\alpha}\pi^{d/2}\Gamma(1-\frac\alpha2)},\quad
    \omega_{d-1}
    :=\frac{2\pi^{d/2}}{\Gamma\left(\frac{d}{2}\right)},
\end{gather*}
which can be estimated as follows (see \cite[Lemma 3.3]{DSX23}):
\begin{gather}\label{crate}
        \frac{A(d,\alpha)\omega_{d-1}}{d(2-\alpha)}\leq C
        \quad\text{and}\quad
        \left| \frac{A(d,\alpha)\omega_{d-1}}{d(2-\alpha)} - 1\right|
        \leq C(2-\alpha)\log(1+d).
\end{gather}

\begin{lemma}\label{genediff}
    Assume \textup{\textbf{(H1)}} and \textup{\textbf{(H2)}}.
    \begin{enumerate}
    \item\label{genediff-a}
    For any $\alpha \in (1,2)$ and $0< s<t\le 1$,
    \begin{align*}
        &\sup_{h\in\Cscr_b(1)} \left\|P_{t-s}(\Ascr^P-\Ascr^Q)Q_{s}h\right\|_{\infty}\\
        &\qquad\leq C(t-s)^{-1/\alpha}\,s^{-1/2} \left(3-\alpha-\frac{1}{3-\alpha}\,s^{1-\frac{\alpha}{2}}\right) d\log(1+d)
        + \frac{C}{\alpha-1}(2-\alpha)s^{-1/2}d.
    \end{align*}

    \item\label{genediff-b}
    For any $\alpha \in (1,2)$, $0< s<t\le 1$ and $k=1,2,\dots$,
    \begin{gather*}
        \sup_{h\in\Cscr_b(1)} \left\|P_{t-s}(\Ascr^P-\Ascr^Q)Q_{s+k}h\right\|_{\infty}
    	\leq
    	C_1(2-\alpha)\left((t-s)^{-1/\alpha}+\frac{1}{\alpha-1}\right) \eup^{-C_2k}d\log(1+d).
    \end{gather*}
    \end{enumerate}
\end{lemma}

\begin{proof}
We begin with some general preparations. Let $h\in\Cscr_b(1)$ and set $f:=Q_{s+k}h$ for $k\in\{0,1,2,\dots\}$. It is not hard to see that the generators of the semigroups $(P_t)_{t\geq 0}$ and $(Q_t)_{t\geq 0}$ are given by
\begin{gather*}
    \Ascr^Pf(x)
    = \langle\nabla f(x),b(x)\rangle
    + \int\limits_{\mathclap{\rd\setminus\{0\}}} \left[f(x+\sigma z)-f(x)-\langle \nabla f(x),\sigma z\rangle \I_{(0,1)}(|z|) \right]
    \frac{A(d,\alpha)}{|z|^{d+\alpha}}\,\dup z,
\intertext{and}
    \Ascr^Qf(x)
    = \langle\nabla f(x),b(x)\rangle + \frac12\,\langle\nabla^2f(x),\sigma\sigma^{\top}\rangle_{\HS}.
\end{gather*}
Therefore,
\begin{align*}
    &\left(\Ascr^P-\Ascr^Q\right)f(x)\\
    &=\left\{ \int_{|z| < 1}\left[f(x+\sigma z)-f(x) - \langle\nabla f(x),\sigma z\rangle\right] \frac{A(d,\alpha)}{|z|^{d+\alpha}}\,\dup z
    - \frac12\,\langle\nabla^2f(x),\sigma\sigma^{\top} \rangle_{\HS}\right\}\\
    &\quad\mbox{}+\int_{|z|\geq 1}\left[f(x+\sigma z)-f(x)\right] \frac{A(d,\alpha)}{|z|^{d+\alpha}}\,\dup z\\
    &=:\mathsf{J}_1(x)+\mathsf{J}_2(x).
\end{align*}
By the first inequality in \eqref{crate}, we get for all $\alpha \in (1,2)$,
\begin{align} \label{jh43dds}
    \begin{split}
    |\mathsf{J}_2(x)|
    &\leq \int_{|z|\geq 1}\left|f(x+\sigma z)-f(x)\right| \frac{A(d,\alpha)}{|z|^{d+\alpha}}\,\dup z\\
    &\leq  \|\nabla f\|_{\infty}\|\sigma\|_{\op}\int_{|z|\geq 1} |z|\,\frac{A(d,\alpha)}{|z|^{d+\alpha}}\,\dup z\\
    &= \|\nabla f\|_{\infty}\|\sigma\|_{\op}\frac{A(d,\alpha)\omega_{d-1}}{\alpha-1}\\
    &\leq \frac{C}{\alpha-1}\,\|\nabla f\|_{\infty}(2-\alpha) d.
\end{split}
\end{align}
We rewrite $\mathsf{J}_1(x)$ in the following form:
\begin{align*}
    &\mathsf{J}_1(x)\\
    &= \int_{|z| < 1}\left[f(x+\sigma z)-f(x) -\langle\nabla f(x),\sigma z\rangle\right] \frac{A(d,\alpha)}{|z|^{d+\alpha}}\,\dup z
    - \frac12\,\langle\nabla^2f(x),\sigma\sigma^{\top}\rangle_{\HS}\\
    &= \int_{|z| < 1}\int_0^1\langle\nabla^2f(x+r\sigma z),(\sigma z)(\sigma z)^{\top}\rangle_{\HS}(1-r)\,\dup r \frac{A(d,\alpha)}{|z|^{d+\alpha}}\,\dup z
    - \frac12\,\langle\nabla^2f(x),\sigma\sigma^{\top}\rangle_{\HS}\\
    &= \left\{\int_{|z| < 1}\int_0^1\langle\nabla^2f(x),(\sigma z)(\sigma z)^{\top}\rangle_{\HS}(1-r)\,\dup r \frac{A(d,\alpha)}{|z|^{d+\alpha}}\,\dup z
    - \frac12\,\langle\nabla^2f(x),\sigma\sigma^{\top}\rangle_{\HS}\right\}\\
    &\qquad\mbox{} + \int_{|z| < 1}\int_0^1\langle\nabla^2f(x+r\sigma z)-\nabla^2f(x), (\sigma z)(\sigma z)^{\top}\rangle_{\HS}(1-r)\,\dup r\, \frac{A(d,\alpha)}{|z|^{d+\alpha}}\,\dup z\\
    &=: \mathsf{J}_{11}(x)+\mathsf{J}_{12}(x).
\end{align*}
Using the symmetry of the measure $\rho(\dif z) = |z|^{-d-\alpha}\,\dif z$, it is clear that $\int_{|z|\leq 1} z_i z_j\,\rho(\dif z) = \delta_{ij} \frac 1d \int_{|z|\leq 1} |z|^2\,\rho(\dif z)$, and so we get
\begin{align*}
    &\int_{|z| < 1}\int_0^1\langle\nabla^2f(x),(\sigma z)(\sigma z)^{\top}\rangle_{\HS}(1-r)\,\dup r\frac{A(d,\alpha)}{|z|^{d+\alpha}}\,\dup z\\
    &\qquad= \frac12\int_{|z| < 1}\langle\nabla^2f(x),\sigma(zz^{\top})\sigma^{\top}\rangle_{\HS}\frac{A(d,\alpha)}{|z|^{d+\alpha}}\,\dup z\\
    &\qquad=\frac12\,\frac{1}{d}\int_{|z| < 1} \langle\nabla^2f(x),|z|^2\sigma\sigma^{\top}\rangle_{\HS}\frac{A(d,\alpha)}{|z|^{d+\alpha}}\,\dup z\\
    &\qquad=\frac{1}{2d}\,\langle\nabla^2f(x),\sigma\sigma^{\top}\rangle_{\HS}\int_{|z|<1}\frac{A(d,\alpha)}{|z|^{d+\alpha-2}}\,\dup z\\
    &\qquad=\frac{A(d,\alpha)\omega_{d-1}}{2d(2-\alpha)}\,\langle\nabla^2f(x),\sigma\sigma^{\top}\rangle_{\HS}.
\end{align*}
Combining this with \eqref{crate}, Proposition \ref{2ndgrad} and the fact that $\|\sigma\sigma^\top\|_\HS\leq\sqrt{d}\|\sigma\sigma^\top\|_\op\leq\sqrt{d}\|\sigma\|_\op^2$, we obtain for all $\alpha \in (1,2)$,
\begin{align} \label{kj654fg}
\begin{split}
    |P_{t-s}\mathsf{J}_{11}(x)|
    &= \left|\frac12\left[\frac{A(d,\alpha)\omega_{d-1}}{d(2-\alpha)}-1\right]\big\langle P_{t-s} \nabla^2f(x),\sigma\sigma^{\top}\big\rangle_{\HS}\right|\\
    &\leq\frac12\left|\frac{A(d,\alpha)\omega_{d-1}}{d(2-\alpha)}-1\right|
    \|P_{t-s}\nabla^{2} f(x)\|_{\HS}\|\sigma\sigma^\top\|_{\HS}\\
    &\leq\frac12\left|\frac{A(d,\alpha)\omega_{d-1}}{d(2-\alpha)}-1\right|
    d\|P_{t-s}\nabla^{2} f(x)\|_{\op,\infty}\|\sigma\|_\op^2\\
    &\leq C(2-\alpha)(t-s)^{-1/\alpha}\|\nabla f\|_\infty  d\log(1+d).
\end{split}
\end{align}
Let us now estimate $P_{t-s}\mathsf{J}_{12}(x)$.
It follows from Proposition \ref{2ndgrad} that for any $r\in[0,1]$ and $\alpha \in (1,2)$,
\begin{align*}
    &\left|\left\langle P_{t-s}\left[\nabla^2\left(f(x+r\sigma z)-f(x)\right) \right],
    (\sigma z)(\sigma z)^{\top}\right\rangle_{\HS} \right|\\
    &\quad\le\left\|P_{t-s}\nabla^2\{f(\cdot+r\sigma z)-f(\cdot)\}(x)\right\|_{\op} |\sigma z|^2\\
    &\quad\le C(t-s)^{-1/\alpha}
    \left\|\nabla\{f(\cdot+r\sigma z)-f(\cdot)\}\right\|_\infty|\sigma z|^2\\
    &\quad\le C(t-s)^{-1/\alpha}
    \left\{[2\|\nabla f\|_\infty]\wedge [\|\nabla^2 f\|_{\op,\infty}r|\sigma z|]\right\} |\sigma z|^2\\
    &\quad\le C(t-s)^{-1/\alpha}
    \left\{\|\nabla f\|_\infty\wedge [\|\nabla^2 f\|_{\op,\infty}|z|]\right\}|z|^2.
\end{align*}
This yields for $\alpha \in (1,2)$,
\begin{align} \label{fgddd34s}
\begin{split}
    &|P_{t-s}\mathsf{J}_{12}(x)|\\
    &\leq\int_{|z| <1}\int_0^1\left|\langle P_{t-s}\nabla^2f(x+r\sigma z)-P_{t-s}\nabla^2f(x),
    (\sigma z)(\sigma z)^{\top}\rangle_{\HS}\right|(1-r)\,\dup r\, \frac{A(d,\alpha)}{|z|^{d+\alpha}}\,\dup z\\
    &\leq C(t-s)^{-1/\alpha} \int_{|z|<1}\left(\int_0^1(1-r)\,\dif r\right)
    \big\{\|\nabla f\|_\infty\wedge [\|\nabla^2 f\|_{\op,\infty}|z|]\big\} |z|^{2}\frac{A(d,\alpha)}{|z|^{d+\alpha}}\,\dup z\\
    &= \frac{C}{2}(t-s)^{-1/\alpha}\int_{|z|<1}\big\{\|\nabla f\|_\infty\wedge [\|\nabla^2 f\|_{\op,\infty}|z|]\big\}\frac{A(d,\alpha)}{|z|^{d+\alpha-2}}\,\dup z.
\end{split}
\end{align}
Now we are ready to prove Lemma \ref{genediff}.

\medskip\noindent
\emph{Proof of \ref{genediff}.\ref{genediff-a}:}
Take in the above calculations $k=0$. Note that $f=Q_{s}h$ and $s\in(0,1)$. By \eqref{grad11}, we get
$\|\nabla f\|_\infty\leq Cs^{-1/2}$. Then by \eqref{jh43dds} and \eqref{kj654fg}, we have for all $\alpha\in (1,2)$,
\begin{gather*}
    |\mathsf{J}_{2}(x)|\leq \frac{C}{\alpha-1}(2-\alpha)s^{-1/2}d,
\end{gather*}
and
\begin{gather*}
    |P_{t-s}\mathsf{J}_{11}(x)|
    \leq C(2-\alpha)(t-s)^{-1/\alpha}s^{-1/2} d\log(1+d),
\end{gather*}
Moreover, by \eqref{grad22}, $\|\nabla^2 f\|_{\op,\infty}\leq Cs^{-1}$. It follows from the first inequality in \eqref{crate} that for all $\alpha\in(0,2)$,
\begin{gather*}
   \frac{A(d,\alpha)\omega_{d-1}}{2-\alpha}\leq C d.
\end{gather*}
Therefore, we obtain from \eqref{fgddd34s} that for $\alpha \in (1,2)$,
\begin{align*}
    |P_{t-s}\mathsf{J}_{12}(x)|
    &\leq C(t-s)^{-1/\alpha}\int_{|z|<1}\left\{s^{-1/2}\wedge [s^{-1}|z|]\right\}
    \frac{A(d,\alpha)}{|z|^{d+\alpha-2}}\,\dup z\\
    &= C(t-s)^{-1/\alpha}
    \left(\int_{|z|\leq \sqrt{s}}s^{-1}|z|\,\frac{A(d,\alpha)}{|z|^{d+\alpha-2}}\,\dup z
    +\int_{\sqrt{s}<|z|< 1}s^{-1/2}\,\frac{A(d,\alpha)}{|z|^{d+\alpha-2}}\,\dup z\right)\\
    &= C(t-s)^{-1/\alpha}\,s^{-1/2}A(d,\alpha)\omega_{d-1}
    \left(s^{-1/2}\int_0^{\sqrt{s}}r^{2-\alpha}\,\dup r +\int_{\sqrt{s}}^1r^{1-\alpha}\,\dup r \right)\\
    &= C(t-s)^{-1/\alpha}\,s^{-1/2}\left(1-\frac{1}{3-\alpha}\,s^{1-\frac{\alpha}{2}}\right) \frac{A(d,\alpha)\omega_{d-1}}{2-\alpha}\\
    &\leq C(t-s)^{-1/\alpha}\,s^{-1/2}\left(1-\frac{1}{3-\alpha}\,s^{1-\frac{\alpha}{2}}\right)d.
\end{align*}
Combining all estimates, we get
for all $\alpha\in (1,2)$, $x\in\rd$ and $0<s<t\leq1$,
\begin{align*}
    &\left|P_{t-s}(\Ascr^P-\Ascr^Q)f(x)\right|\\
    &\quad\leq |P_{t-s}\mathsf{J}_{11}(x)| + |P_{t-s}\mathsf{J}_{12}(x)|+|P_{t-s}\mathsf{J}_{2}(x)|\\
    &\quad\leq C(t-s)^{-1/\alpha}\,s^{-1/2}
    \left[(2-\alpha)
    +\left(
    1-\frac{1}{3-\alpha}\,s^{1-\frac{\alpha}{2}}
    \right)
    \right]
    d\log(1+d)
    +\frac{C}{\alpha-1}(2-\alpha)s^{-1/2}d,
\end{align*}
from which we immediately get the first part of the lemma.

\medskip\noindent
\emph{Proof of \ref{genediff}.\ref{genediff-b}:}
 Let $k\geq1$. Since $f=Q_{s+k}h$ and $s+k>1$, it follows from Lemma \ref{TV-gradient} that
\begin{gather*}
    \|\nabla f\|_\infty\leq C_1\eup^{-C_2k}
    \quad\text{and}\quad
    \|\nabla^2 f\|_{\op,\infty}\leq C_1\eup^{-C_2k}.
\end{gather*}
Using \eqref{jh43dds} and \eqref{kj654fg} implies for all $\alpha\in(1,2)$,
\begin{gather*}
    |\mathsf{J}_{2}(x)|\leq \frac{C_1}{\alpha-1}(2-\alpha)\eup^{-C_2k}d,
\intertext{and}
    |P_{t-s}\mathsf{J}_{11}(x)|
    \leq C_1(2-\alpha)(t-s)^{-1/\alpha}\eup^{-C_2k} d\log(1+d).
\end{gather*}
With the first estimate in \eqref{crate}, it is easy to check that for any $\alpha \in (0,2)$ the following estimate holds:
\begin{gather*}
    \frac{A(d,\alpha)\omega_{d-1}}{3-\alpha}\leq C(2-\alpha)d.
\end{gather*}
Using \eqref{fgddd34s}, we get for all $\alpha\in(1,2)$,
\begin{align*}
    |P_{t-s}\mathsf{J}_{12}(x)|
    &\leq C_1(t-s)^{-1/\alpha}\int_{|z|\leq1}\eup^{-C_2k}|z|
    \frac{A(d,\alpha)}{|z|^{d+\alpha-2}}\,\dup z\\
    &= C_1(t-s)^{-1/\alpha}\eup^{-C_2k}
    A(d,\alpha)\omega_{d-1}\int_0^1r^{2-\alpha}\,\dif r\\
    &=C_1(t-s)^{-1/\alpha}\eup^{-C_2k}
    \frac{A(d,\alpha)\omega_{d-1}}{3-\alpha}\\
    &\leq C_1(2-\alpha)(t-s)^{-1/\alpha}
    \eup^{-C_2k}d.
\end{align*}
Finally, combining all estimates, we conclude that for all $\alpha\in (1,2)$,
\begin{align*}
    \left|P_{t-s}(\Ascr^P-\Ascr^Q)f(x)\right|
    &\leq |P_{t-s}\mathsf{J}_{11}(x)| + |P_{t-s}\mathsf{J}_{12}(x)|+|P_{t-s}\mathsf{J}_{2}(x)|\\
    &\leq C_1(2-\alpha)(t-s)^{-1/\alpha}
    \eup^{-C_2k}d\log(1+d)+\frac{C_1}{\alpha-1}(2-\alpha)\eup^{-C_2k}d,
\end{align*}
which completes the proof of the second part of the lemma.
\end{proof}

\section{Proofs of Theorem \ref{main1} of Corollary \ref{corollary}}
\begin{lemma} \label{beta}
    For any $t\in(0,1]$ and $\alpha\in (1,2)$,
    \begin{align*}
        \int_{0}^{t} (t-s)^{-\frac{1}{\alpha}}\,s^{-\frac12}
        \left(3-\alpha -\frac{1}{3-\alpha}\,s^{1-\frac{\alpha}{2}} \right)\dif s
        \le  \frac{C}{\alpha-1} (2-\alpha) t^{-\frac{1}{\alpha}+\frac{1}{2}}\left(1+\log \tfrac{1}{t}\right).
    \end{align*}
\end{lemma}
\begin{proof}
    Using the elementary estimate
    \begin{gather*}
        1-s^\kappa=-\log s\int_0^\kappa s^x\,\dif x
        \leq(-\log s)\kappa=\kappa\log\tfrac1s,
        \qquad 0<s<1,\; \kappa>0,
    \end{gather*}
    yields for all $s\in (0,1)$ and $\alpha\in(0,2)$,
    \begin{align*}
        3-\alpha-\frac{s^{1-\frac{\alpha}{2}}}{3-\alpha}
        &=\frac{(4-\alpha)(2-\alpha)}{3-\alpha}+\frac{1}{3-\alpha}\left(1-s^{1-\frac{\alpha}{2}} \right)\\
        &<2(2-\alpha)+\left(1-\tfrac\alpha2\right)\log\tfrac{1}{s}\\
        &<2(2-\alpha)\left(1+\log\tfrac{1}{s}\right).
    \end{align*}
    This, together with the change of variables $s=t\tau$, implies that for any $t\in (0,1]$ and $\alpha\in(1,2)$,
    \begin{align} \label{eqqchange}
    \begin{split}
        \int_{0}^{t} &(t-s)^{-\frac{1}{\alpha}}\,s^{-\frac12}
        \left(3-\alpha -\frac{1}{3-\alpha}\,s^{1-\frac{\alpha}{2}} \right)\dif s\\
        &\leq 2(2-\alpha)\int_{0}^{t} (t-s)^{-\frac{1}{\alpha}}\,s^{-\frac12}
        \left(1+\log\tfrac{1}{s}\right)\dif s\\
        &=2(2-\alpha)t^{-\frac{1}{\alpha}+\frac{1}{2}} \int_{0}^{1} (1-\tau)^{-\frac{1}{\alpha}} \tau^{-\frac{1}{2}} \left(
        1+\log\tfrac{1}{t\tau}\right)\dif \tau\\
        &=2(2-\alpha)t^{-\frac{1}{\alpha}+\frac{1}{2}}\left(
        \int_{0}^{1} (1-\tau)^{-\frac{1}{\alpha}} \tau^{-\frac{1}{2}}
        \left(
        1+\log\tfrac{1}{\tau}\right)\dif \tau+\operatorname{B}\left(1-\tfrac{1}{\alpha},\tfrac{1}{2}\right)\log\tfrac1t
        \right).
    \end{split}
    \end{align}
    As usual, $\operatorname{B}(\cdot,\cdot)$ is Euler's Beta function.
    From $\Gamma(x)<1/x$ for $x\in(0,1)$, we see that for all $\alpha\in(1,2)$,
    \begin{equation}\label{gammax}
        \operatorname{B}\left(1-\tfrac{1}{\alpha},\tfrac{1}{2}\right)
        =\frac{\Gamma\left(1-\frac{1}{\alpha}\right)\Gamma\left(\frac{1}{2}\right)}
        {\Gamma\left(\frac32-\frac{1}{\alpha}\right)}
        < \sqrt{\pi}\,\Gamma\left(1-\tfrac{1}{\alpha}\right)
        < \sqrt{\pi}\,\frac{\alpha}{\alpha-1}
        < \frac{2\sqrt{\pi}}{\alpha-1}.
    \end{equation}
    Noting that there exists some constant $C>0$ such that
    $$
        1+\log\tfrac{1}{\tau}\leq C\tau^{-1/4} \quad\text{for all $\tau\in(0,1)$},
    $$
    as in \eqref{gammax} we have for all $\alpha\in(1,2)$,
    \begin{align*}
        \int_{0}^{1} (1-\tau)^{-\frac{1}{\alpha}} &\tau^{-\frac{1}{2}}
        \left(
        1+\log\tfrac{1}{\tau}\right)\dif \tau
        \leq C\int_{0}^{1} (1-\tau)^{-\frac{1}{\alpha}} \tau^{-\frac{3}{4}}\,\dif \tau\\
        &=C\operatorname{B}\left(1-\tfrac{1}{\alpha},\tfrac{1}{4}\right)
        =C\frac{\Gamma\left(1-\frac{1}{\alpha}\right)\Gamma\left(\frac{1}{4}\right)}
        {\Gamma\left(\frac54-\frac{1}{\alpha}\right)}\\
        &<C\frac{\Gamma\left(\frac{1}{4}\right)}{\Gamma\left(\frac{3}{4}\right)}\,\Gamma\left(1-\tfrac{1}{\alpha}\right)
        <\frac{\Gamma\left(\frac{1}{4}\right)}{\Gamma\left(\frac{3}{4}\right)}\,\frac{2C}{\alpha-1}.
    \end{align*}
    Substituting this and \eqref{gammax} into \eqref{eqqchange}, finishes the proof.
\end{proof}

\begin{proposition}\label{geneint}
    Assume \textup{\textbf{(H1)}} and \textup{\textbf{(H2)}}. For all $\alpha \in (1,2)$ and $t>0$,
    \begin{gather*}
        \sup_{x\in\real^d}\|\law(X_t^x)-\law(Y_t^x)\|_\TV
        \leq \frac{Cd^{3/2}}{\alpha-1}(2-\alpha)\cdot
        \left\{\left(t^{-\frac{1}{\alpha}+\frac{1}{2}}\log\tfrac{1}{t}\right) \vee 1  \right\}.
    \end{gather*}
\end{proposition}

\begin{proof}
Recall that $\Cscr_b(1)$ is the unit ball in $C_b(\rd,\real)$. Since
\begin{gather*}
    \|\law(X_t^x)-\law(Y_t^x)\|_\TV=\sup_{h\in \Cscr_{b}(1)}|P_t h(x)-Q_t h(x)|
\intertext{and}
    P_t h-Q_t h
    = -\int_0^t \frac{\dif}{\dif s}\,P_{t-s} Q_s h \,\dif s
    =\int_0^t P_{t-s} (\Ascr^P-\Ascr^Q)Q_s h\,\dif s,
\intertext{we get}
    \sup_{x\in\real^d}\|\law(X_t^x)-\law(Y_t^x)\|_\TV=\sup_{h\in \Cscr_{b}(1)}\sup_{x\in\real^d}\left|
    \int_0^t P_{t-s} (\Ascr^P-\Ascr^Q)Q_s h(x)\,\dif s
    \right|.
\end{gather*}
From Lemmas \ref{genediff} and \ref{beta} we see that for all $\alpha\in(1,2)$ and $t\in(0,1]$,
\begin{align} \label{ffs:a11}
\begin{split}
    \sup_{h\in \Cscr_{b}(1)}&\int_0^t \left\|P_{t-s} (\Ascr^P-\Ascr^Q)Q_s h\right\|_\infty\dif s\\
    &\leq Cd\log(1+d)\cdot\int_{0}^{t} (t-s)^{-\frac{1}{\alpha}}\,s^{-\frac12}
        \left(3-\alpha -\frac{1}{3-\alpha}\,s^{1-\frac{\alpha}{2}} \right)\dif s\\
        &\quad\mbox{}+\frac{Cd}{\alpha-1}(2-\alpha)\int_0^ts^{-1/2}\,\dif s\\
        &\leq\frac{Cd\log(1+d)}{\alpha-1}(2-\alpha)t^{-\frac{1}{\alpha}+\frac{1}{2}}\left(1+\log \tfrac{1}{t}\right)
        +\frac{2Cd}{\alpha-1}(2-\alpha)t^{1/2}\\
        &\leq\frac{C_1d\log(1+d)}{\alpha-1}(2-\alpha)t^{-\frac{1}{\alpha}+\frac{1}{2}}\left(1+\log \tfrac{1}{t}\right).
\end{split}
\end{align}
Combining this with
\begin{gather*}
    \sup_{x\in\real^d}\|\law(X_t^x)-\law(Y_t^x)\|_\TV
    \leq\sup_{h\in \Cscr_{b}(1)}\int_0^t \left\|P_{t-s} (\Ascr^P-\Ascr^Q)Q_s h\right\|_\infty\dif s,
\end{gather*}
proves assertion for $t\in(0,1]$.

Next, we consider the case $t>1$. Denote by $\lfloor t\rfloor$ the largest integer which is less than or equal to $t>1$. From
\begin{align*}
    &\int_0^t P_{t-s} (\Ascr^P-\Ascr^Q)Q_s h\,\dif s\\
    &=\sum_{k=0}^{\lfloor t\rfloor-1}\int_{k}^{k+1}P_{t-k-1}P_{k+1-s}
    (\Ascr^P-\Ascr^Q)Q_s h\,\dif s
    +\int_{\lfloor t\rfloor}^tP_{t-s}
    (\Ascr^P-\Ascr^Q)Q_s h\,\dif s\\
    &=\sum_{k=0}^{\lfloor t\rfloor-1}\int_{0}^{1}P_{t-k-1}P_{1-s}
    (\Ascr^P-\Ascr^Q)Q_{s+k} h\,\dif s
    +\int_{0}^{t-\lfloor t\rfloor}P_{t-\lfloor t\rfloor-s}
    (\Ascr^P-\Ascr^Q)Q_{\lfloor t\rfloor +s} h\,\dif s,
\end{align*}
we obtain
\begin{align} \label{ffs3w2d}
\begin{split}
    \sup_{x\in\real^d}&\|\law(X_t^x)-\law(Y_t^x)\|_\TV\\
    &\leq\sup_{h\in\Cscr_b(1)}\int_{0}^{1}
    \left\|P_{1-s}(\Ascr^P-\Ascr^Q)Q_{s} h\right\|_\infty\dif s\\
    &\quad\mbox{}+
    \sum_{k=1}^{\lfloor t\rfloor-1}\sup_{h\in\Cscr_b(1)}\int_{0}^{1}
    \left\|P_{1-s}(\Ascr^P-\Ascr^Q)Q_{s+k} h\right\|_\infty\dif s\\
    &\quad\mbox{}+\sup_{h\in\Cscr_b(1)}\int_{0}^{t-\lfloor t\rfloor}
    \left\|P_{t-\lfloor t\rfloor-s}(\Ascr^P-\Ascr^Q)Q_{\lfloor t\rfloor +s} h\right\|_\infty\dif s,
\end{split}
\end{align}
where we use the `empty-sum-convention' $\sum_{k=1}^0\dots=0$. Using \eqref{ffs:a11} with $t=1$, we see that for all $\alpha\in(1,2)$,
\begin{gather*}
    \sup_{h\in\Cscr_b(1)}\int_0^1 \left\|P_{1-s} (\Ascr^P-\Ascr^Q)Q_s h\right\|_{\infty}\dif s
    \leq \frac{Cd\log(1+d)}{\alpha-1}(2-\alpha).
\end{gather*}
By Lemma \ref{genediff}.\ref{genediff-b}, we find for all $\alpha \in (1,2)$,
\begin{align*}
    \sup_{h\in\Cscr_b(1)}&\int_{0}^{1} \left\|P_{1-s}(\Ascr^P-\Ascr^Q)Q_{s+k} h\right\|_\infty\dif s\\
    &\leq C_1(2-\alpha)\eup^{-C_2k}\int_0^1\left( (1-s)^{-1/\alpha}+\frac{1}{\alpha-1}\right)\dif s\cdot d\log(1+d)\\
    &=C_1(2-\alpha)\eup^{-C_2k}\left(\frac{\alpha}{\alpha-1}+\frac{1}{\alpha-1}\right)d\log(1+d)\\
    &\leq\frac{3C_1d\log(1+d)}{\alpha-1}(2-\alpha)\eup^{-C_2k},
\end{align*}
and
\begin{align*}
    \sup_{h\in\Cscr_b(1)}&\int_{0}^{t-\lfloor t\rfloor}
    \left\|P_{t-\lfloor t\rfloor-s}(\Ascr^P-\Ascr^Q)Q_{s+\lfloor t\rfloor } h\right\|_\infty\dif s\\
    &\leq C_1(2-\alpha)\eup^{-C_2\lfloor t\rfloor}
    \int_{0}^{t-\lfloor t\rfloor}\left((t-\lfloor t\rfloor-s)^{-1/\alpha}+\frac{1}{\alpha-1}\right)\dif s\cdot d\log(1+d)\\
    &=C_1(2-\alpha)\eup^{-C_2\lfloor t\rfloor}
    \left((t-\lfloor t\rfloor)^{1-1/\alpha}\frac{\alpha}{\alpha-1}+\frac{1}{\alpha-1}(t-\lfloor t\rfloor)\right)d\log(1+d)\\
    &\leq C_1(2-\alpha)\left(\frac{\alpha}{\alpha-1}+\frac{1}{\alpha-1}\right)d\log(1+d)\\
    &\leq\frac{3C_1d\log(1+d)}{\alpha-1}(2-\alpha).
\end{align*}
Substituting these estimates into \eqref{ffs3w2d}, shows that for all $t\geq1$ and $\alpha\in(1,2)$,
\begin{align*}
    \sup_{x\in\real^d}\|\law(X_t^x)-\law(Y_t^x)\|_\TV
    &\leq \frac{C_1d\log(1+d)}{\alpha-1}(2-\alpha)\left(
    1+\sum_{k=1}^{\lfloor t\rfloor-1}\eup^{-C_2k}
    \right)\\
    &\leq \frac{Cd\log(1+d)}{\alpha-1}(2-\alpha).
\end{align*}
This completes the proof.
\end{proof}

\begin{proposition}\label{dd3hj3s}
    Assume \textup{\textbf{(H1)}} and \textup{\textbf{(H2)}}. For all $x,y\in \rd$ and $t>0$,
    \begin{gather*}
        \left\|\law(Y^{x}_{t})-\law(Y^{y}_{t})\right\|_{\TV}
        \leq C_{1} \left\{t^{-1/2} \vee 1\right\}\eup^{-C_{2}t}|x-y|.
    \end{gather*}
\end{proposition}

\begin{proof}
    By the definition of the total variation norm $\|\cdot\|_\TV$, one has
    \begin{gather*}
        \left\|\law(Y^{x}_{t})-\law(Y^{y}_{t})\right\|_{\TV}
        =\sup_{h\in \Cscr_b(1)}\left|Q_th(x)-Q_th(y)\right|
        \leq\sup_{h\in \Cscr_b(1)}\left\|\nabla Q_th\right\|_\infty|x-y|.
    \end{gather*}
    Together with \eqref{grad11} we get the claimed estimate.
\end{proof}

\bigskip
\begin{proof}[\bfseries Proof of Theorem \ref{main1}]
Note that
\begin{gather*}
     \left\|\law (X_{t}^x)-\law (Y_{t}^y)\right\|_\TV
     \leq\left\|\law (X_{t}^x)-\law (Y_{t}^x)\right\|_\TV+
     \left\|\law (Y_{t}^x)-\law (Y_{t}^y)\right\|_\TV.
\end{gather*}
Therefore, the first assertion follows straight from Propositions \ref{geneint} and \ref{dd3hj3s}.

From the classical ergodic theory for Markov processes (see, for instance, \cite{Wa13}), one has
\begin{gather*}
    \lim_{t\to\infty}\left\|\law (X_{t}^x)-\mu_\alpha\right\|_\TV
    =\lim_{t\to\infty}\left\|\law (Y_{t}^y)-\mu_2\right\|_\TV
    =0,\quad x,y\in\rd.
\intertext{Since for any fixed $x,y\in \rd$}
    \left\|\mu_\alpha-\mu_2\right\|_\TV
    \leq \left\|\mu_\alpha-\law(X_{t}^x)\right\|_\TV + \left\|\law (X_{t}^x)-\law (Y_{t}^y)\right\|_\TV
    + \left\|\law (Y_{t}^y)-\mu_2\right\|_\TV,
\end{gather*}
we can let $t\to\infty$ to get \eqref{was-est} from the first part.
\end{proof}

\begin{proof}[\bfseries Proof of Corollary \ref{corollary}]
 It follows from \cite[Lemma 2.1]{HRW23} with $\psi(r)=r^{p}$ that for any $p\in(0,1)$ and probability measures $\mu,\nu$,
   \begin{align*}
       \|\mu-\nu\|_{ W_p}&\leq\inf_{t>0}\left\{
        t^{p/2}\sqrt{d}\,\|\mu-\nu\|_\TV+t^{(p-1)/2}d\,\|\mu-\nu\|_{ W_1}
        \right\}\\
        &\leq d^{(p+1)/2}\left\{
        \|\mu-\nu\|_\TV+\|\mu-\nu\|_{ W_1}
        \right\}.
   \end{align*}
   If \textup{\textbf{(H1)}} and \textup{\textbf{(H2)}} hold, we know from \cite[Theorem 1.1]{DSX23} that
   for any $\alpha\in(1,2)$,
   \begin{gather*}
        \|\mu_\alpha-\mu_2\|_{W_1}\leq\frac{Cd\log(1+d)}{\alpha-1} (2-\alpha).
   \end{gather*}
   Combining these two estimates with \eqref{was-est}, we obtain the desired estimate.
\end{proof}

\section{Optimality for the Ornstein--Uhlenbeck process on $\real$} \label{s:OU}

In this section, we assume that $\mu_\alpha$ is the invariant measure of the Ornstein--Uhlenbeck process on $\real$, which solves the following SDE:
\begin{gather*}
    \dif X_{t}=-X_{t}\,\dif t+\dif L_{t}, \quad X_{0}=x\in\real.
\end{gather*}
Here, $L_{t}$ is a rotationally symmetric $\alpha$-stable L\'evy process on $\real$ with $\Ee\left[\eup^{\iup\xi L_t}\right]=\eup^{-t|\xi|^\alpha/2}$, and $\mu_2$ is the invariant measure of
\begin{gather*}
    \dif Y_{t}=-Y_{t}\,\dif t+\dif B_{t}, \quad Y_{0}=y\in\real,
\end{gather*}
where $B_{t}$ is a standard Brownian motion on $\real$ with $\Ee\left[\eup^{\iup\xi B_t}\right]=\eup^{-t|\xi|^2/2}$.

It is easy to see that \textbf{(H1)} holds with $\sigma=1$, while \textbf{(H2)}  holds with $\theta_0=1$, $K=0$, $\theta_1=1$ and $\theta_2=0$.

\begin{proposition}\label{ourate}
We have
\begin{gather*}
    0 < \liminf_{\alpha\uparrow2}\frac{\|\mu_\alpha-\mu_2\|_\TV}{2-\alpha}
    \leq \limsup_{\alpha\uparrow2}\frac{\|\mu_\alpha-\mu_2\|_\TV}{2-\alpha}
    < \infty.
\end{gather*}
\end{proposition}

\begin{proof}
The upper bound follows from Theorem \ref{main1}. For the lower bound we use that
$X_{t}=\eup^{-t}x+\eup^{-t}\int_{0}^{t}\eup^{s}\,\dif L_{s}$. For $\xi\in\real$ we see
\begin{align*}
    \Ee\left[\eup^{\iup\xi X_{t}}\right]
    &= \eup^{\iup\xi\eup^{-t}x} \Ee\left[ \eup^{\iup\int_{0}^{t}\xi \eup^{-t}\eup^{s}\,\dif L_{s}}\right]
    = \eup^{\iup\xi\eup^{-t}x}\eup^{-2^{-1}\int_{0}^{t}|\xi \eup^{-t}\eup^{s}|^{\alpha}\,\dif s}\\
    &= \eup^{\iup\xi\eup^{-t}x} \eup^{-(2\alpha)^{-1}|\xi|^{\alpha}(1-\eup^{-\alpha t})}
    \;\xrightarrow[]{\; t\to\infty\; }\;
    \eup^{-|\xi|^{\alpha}/(2\alpha)}
    =\Ee\left[\eup^{\iup\xi\alpha^{-1/\alpha}L_{1}}\right].
\end{align*}
Thus, the ergodic measure $\mu_\alpha$ is given by the law of $\alpha^{-1/\alpha}L_{1}$.  Similarly, $\mu_2=\law(2^{-1/2}B_{1})$.

Let
\begin{gather*}
    h(x):=\cos x=\RE\eup^{\iup x},\quad x\in\real.
\end{gather*}
Since $|h(x)|\leq 1$, it follows that
\begin{align*}
    \|\mu_2-\mu_\alpha\|_\TV
    &\geq \mu_2(h)-\mu_\alpha(h)\\
    &= \RE\int_\real\eup^{\iup x}\,\Pp\left(2^{-1/2}B_{1}\in\dup x\right)
    - \RE\int_\real\eup^{\iup x}\,\Pp\left(\alpha^{-1/\alpha}L_{1}\in\dup x\right)\\
    &= \Ee\left[\eup^{\iup2^{-1/2}B_{1}}\right] - \Ee\left[\eup^{\iup\alpha^{-1/\alpha}L_{1}}\right]\\
    &=\eup^{-1/4}-\eup^{-1/(2\alpha)},
\end{align*}
which yields
\begin{gather*}
    \liminf_{\alpha\uparrow2}\frac{\|\mu_2-\mu_\alpha\|_\TV}{2-\alpha}
    \geq \lim_{\alpha\uparrow2}
    \frac{\eup^{-1/4} - \eup^{-1/(2\alpha)}}{2-\alpha}
    =\frac18\,\eup^{-1/4}.
\qedhere
\end{gather*}
\end{proof}

\section{Appendix}

The following exponential contraction result is due to Eberle~\cite[Corollary 2]{E16}, but in concrete situations it is very hard to verify the rather general conditions of that paper without examining the proofs. Therefore, and for the convenience of our readers, we decided to give a short proof of our own, based on Eberle's ideas, which is very easy to apply in the present situation.

\begin{proposition}\label{expcontra}
	Consider the SDE \eqref{BM-SDE} and assume \textup{\textbf{(H1)}}, \eqref{H1-1} and \eqref{H1-1'}. There exist constants $C_1,C_2>0$ such that for all $x,y\in{\mathds{R}^d}$ and $t\geq 0$,
	\begin{gather*}
		\left\|\law(Y^{x}_{t})-\law(Y^{y}_{t})\right\|_{W_1}
		\leq C_1\eup^{-C_{2}t}|x-y|.
	\end{gather*}
\end{proposition}

We begin with a few preparations. Set $R_0:=(2K/\theta_0)^{1/2}$. It follows from \eqref{H1-1} and \eqref{H1-1'} that for all $x,y\in{\mathds{R}^d}$,
\begin{equation}\label{onesidedlip}
	\langle x-y, b(x)-b(y)\rangle\leq\phi(|x-y|)|x-y|,
\end{equation}
where
\begin{gather*}
	\phi(r):=\theta_1r\I_{[0,R_0]}(r) - \frac{1}{2}\theta_0 r\I_{(R_0,\infty)}(r),\quad r\geq0.
\end{gather*}
Without loss of generality, we may assume that the matrix $\sigma\in\real^{d\times d}$ is symmetric and positive definite; otherwise, we could replace $\sigma$ by $(\sigma\sigma^\top)^{1/2}$. Denote by $\lambda_{\min}$ the smallest eigenvalue of $\sigma$. For $r\geq 0$, let
\begin{gather*}
	f(r)
	:=\int_0^r\exp\left[
	-\frac{\theta_0+2\theta_1}{8\lambda_{\min}}(s\wedge R_0)^2
	\right]\dif s.
\end{gather*}
It is easy to verify that $f$ is piecewise of class $C^2$, $f'>0$ is continuous and decreasing on $[0,\infty)$, and
\begin{equation}\label{linearbound}
	f'(R_0)r
	\leq f(r)
	\leq r,\quad r\geq 0.
\end{equation}

\begin{lemma}\label{dissibound}
	For all $r\in[0,\infty)\setminus\{R_0\}$,
	\begin{gather*}
		f'(r)\phi(r)+2\lambda_{\min}f''(r)
		\leq -\frac{\theta_0}{2}f'(R_0)f(r).
	\end{gather*}
\end{lemma}

\begin{proof}
	If $r\in[0,R_0)$, we get by the definition of $f$, the monotonicity of $f'$ and the second inequality in \eqref{linearbound},
	\begin{align*}
		f'(r)\phi(r)+2\lambda_{\min}f''(r)
		&=\theta_1rf'(r)+2\lambda_{\min}f''(r)\\
		&=-\frac{\theta_0}{2}rf'(r)
		\leq-\frac{\theta_0}{2}f'(R_0)r
		\leq-\frac{\theta_0}{2}f'(R_0)f(r).
	\end{align*}
	If $r>R_0$, then $f'(r)=f'(R_0)$ and $f''(r)=0$, and we get
	\begin{gather}
		f'(r)\phi(r)+2\lambda_{\min}f''(r)
		=f'(R_0)\cdot\left(-\frac{\theta_0}{2}r\right)
		\leq-\frac{\theta_0}{2}f'(R_0)f(r).
	\end{gather}
\end{proof}

\begin{proof}[Proof of Proposition \ref{expcontra}]
	Fix $x,y\in{\mathds{R}^d}$ and assume that $x\neq y$. Consider the following SDE with initial value $\hat{Y}_0=y$
	\begin{gather*}
		\left\{\begin{aligned}
			\dif \hat{Y}_t &=b(\hat{Y}_t)\,\dif t+\sigma\left(
			I_{d\times d}-2e_te_t^\top
			\right)\dif B_t,&&t<\tau,\\
			\hat{Y}_t &=Y_t^x,&&t\geq\tau,
		\end{aligned}\right.
	\end{gather*}
	where $I_{d\times d}\in\real^{d\times d}$ is the identity matrix,  $e_t:= \frac{\sigma^{-1}(Y_t^x-\hat{Y}_t)}{|\sigma^{-1}(Y_t^x-\hat{Y}_t)|}$ and $\tau:=\inf\left\{t\geq 0\,:\,\hat{Y}_t=Y_t^x\right\}$ is the coupling time. Since the process
	$(I_{d\times d}-2e_te_t^\top)_{t\geq 0}$ takes values in the orthogonal $d\times d$-matrices, we know from L\'{e}vy's criterion that
	\begin{gather*}
		\int_0^t\left(
		I_{d\times d}-2e_se_s^\top\I_{\{s<\tau\}}
		\right)\dif B_s,\quad t\geq0
	\end{gather*}
	is a standard Brownian motion in ${\mathds{R}^d}$. Therefore, $\law(\hat{Y}_t)=\law(Y_t^y)$ for all $t\geq 0$.
	Set $Z_t:=Y_t^x-\hat{Y}_t$ and note that $(W_t)_{t\geq 0}$, $W_t:=\int_0^t e_s^\top\,\dif B_s$, is a standard
	Brownian motion in $\real$. By the It\^o-Meyer formula, we get for all $t<\tau$
	\begin{gather*}
		\dif |Z_t|
		= \frac{1}{|Z_t|}\left\langle Z_t, b(Y_t^x)-b(\hat{Y}_t)\right\rangle \dif t
		+\frac{2|Z_t|}{|\sigma^{-1}Z_t|}\,\dif W_t
	\end{gather*}
	(note that the local-time part vanishes as $t<\tau$, i.e.\ $Z_t\neq 0$). Since $f$ is piecewise $C^2$ and $f'$ is continuous, we can use It\^o's formula and the expression for $\dif |Z_t|$ to get for all $t<\tau$,
	\begin{gather*}
		\dif f(|Z_t|)
		= \frac{f'(|Z_t|)}{|Z_t|}\left\langle Z_t, b(Y_t^x)-b(\hat{Y}_t)\right\rangle \dif t
		+ f''(|Z_t|)\frac{2|Z_t|^2}{|\sigma^{-1}Z_t|^2}\,\dif t
		+ f'(|Z_t|)\frac{2|Z_t|}{|\sigma^{-1}Z_t|}\,\dif W_t.
	\end{gather*}
	Since $f''\leq0$ on $[0,\infty)\setminus\{R_0\}$ and $|\sigma^{-1}x| \leq |x|/\lambda_{\min}$ for $x\in{\mathds{R}^d}$, it follows from \eqref{onesidedlip} and Lemma \ref{dissibound} that
	\begin{align*}
		\dif f(|Z_t|)
		&\leq \left[f'(|Z_t|)\phi(|Z_t|)+2\lambda_{\min}f''(|Z_t|)\right]\dif t
		+ f'(|Z_t|)\frac{2|Z_t|}{|\sigma^{-1}Z_t|}\,\dif W_t\\
		&\leq-Cf(|Z_t|)\,\dif t+f'(|Z_t|)\frac{2|Z_t|}{|\sigma^{-1}Z_t|}\,\dif W_t,
	\end{align*}
	where $C = \frac 12\theta_0f'(R_0)$. This implies for $t<\tau$ that
	\begin{gather*}
		\dif\left[\eup^{Ct}f(|Z_t|)\right]
		\leq \eup^{Ct} f'(|Z_t|) \frac{2|Z_t|}{|\sigma^{-1}Z_t|}\,\dif W_t.
	\end{gather*}
	Using the finite stopping times $\tau_n := \inf\left\{t\geq0\,:\,|Z_t|\notin[1/n,n]\right\} \leq \tau$ and optional stopping, we see that the right-hand side of the above inequality is a mean-zero martingale. Hence,
	\begin{gather*}
		\Ee\left[\eup^{C(t\wedge\tau_n)}f(|Z_{t\wedge\tau_n}|)\right]\leq f(|Z_0|)=f(|x-y|),\quad t\geq 0.
	\end{gather*}
	Since $\tau_n\uparrow\tau$ as $n\uparrow\infty$, we can apply Fatou's lemma and get
	\begin{gather*}
		\Ee\left[\eup^{C(t\wedge\tau)}f(|Z_{t\wedge\tau}|)\right]
		\leq f(|x-y|).
	\end{gather*}
	Noting that $Z_t=0$ for $t\geq\tau$ and $f(0)=0$, we conclude that
	\begin{gather*}
		\eup^{Ct}\Ee f(|Z_t|)\leq f(|x-y|).
	\end{gather*}
	Because of \eqref{linearbound}, we see that for all $t\geq0$,
	\begin{gather*}
		f'(R_0)\Ee|Y_t^x-\hat{Y}_t|\leq\Ee f(|Z_t|)
		\leq\eup^{-Ct}f(|x-y|)
		\leq\eup^{-Ct}|x-y|.
	\end{gather*}
	Since, by construction, $\law(\hat{Y}_t)=\law(Y_t^y)$, the proof is finished.
\end{proof}

\begin{ack}
C.-S.\ Deng is supported by  the National Natural Science Foundation of China (12371149) and the Alexander von Humboldt Foundation.
R.L.\ Schilling is supported by the 6G-Life project (BMBF Grant No.~16KISK001K) and TU Dresden's ScaDS.AI initiative.
L.\ Xu is supported by the National Natural Science Foundation of China No. 12071499, The Science and Technology Development Fund (FDCT) of Macau S.A.R. FDCT 0074/2023/RIA2, and the University of Macau grants MYRG2020-00039-FST, MYRG-GRG2023-00088-FST.
\end{ack}


\end{document}